\documentclass[12pt,a4paper]{article}
\usepackage[utf8]{inputenc}
\usepackage{amsmath}
\usepackage{amsfonts}
\usepackage{amssymb}
\usepackage{amsthm}
\usepackage{color}
\usepackage{natbib}
\usepackage[left=1.1cm, right=1.1cm]{geometry}
\usepackage{pgf,tikz}
\usepackage[ruled,noend]{algorithm2e}
\usepackage{graphicx}
\usepackage{subfig}
\usepackage{xr-hyper}
\usepackage{placeins}
\usepackage{hyperref}
\usepackage{authblk}

\newtheorem{definition}{Definition}
\newtheorem{proposition}{Proposition}

\newtheorem{lemma}{Lemma}
\newtheorem{theorem}{Theorem}

\DeclareMathOperator{\pa}{pa}

\DeclareMathOperator{\an}{an}
\DeclareMathOperator{\de}{de}

\DeclareMathOperator{\cn}{cn}
\DeclareMathOperator{\nd}{nd}
\DeclareMathOperator{\forb}{forb}
\DeclareMathOperator{\ignore}{ignore}

\begin{document}

\title{A note on efficient minimum cost adjustment sets in causal graphical models}
\author[1]{Ezequiel Smucler \thanks{ezequiels.90@gmail.com}}
\affil[1]{Glovo}
\author[2]{Andrea Rotnitzky \thanks{arotnitzky@utdt.edu}}
\affil[2]{Universidad Torcuato Di Tella and Harvard T.H. Chan School of Public Health}

\maketitle

\begin{abstract}
We study the selection of adjustment sets for estimating the interventional
mean under an individualized treatment rule.  We assume a non-parametric
causal graphical model with, possibly, hidden variables and at least one
adjustment set comprised of observable variables.  Moreover, we assume that observable variables have positive costs associated with them.   We define the cost of an observable adjustment set as the sum of the costs of the variables that comprise it.
We show that in this setting there exist adjustment sets that are minimum cost optimal,  in the
sense that they yield non-parametric estimators of the interventional mean with the
smallest asymptotic variance  among those that
control for observable adjustment sets that have minimum cost.
Our results are based on the construction of a special flow network associated with the original causal graph.  We show that a minimum cost optimal adjustment set can be found by computing a maximum flow on the network,  and then finding the set of vertices that are reachable from the source by augmenting paths.  The \texttt{optimaladj} Python package implements the algorithms introduced in this paper.
\end{abstract}

\section{Introduction}

This paper contributes to the growing literature on graphical criteria for
selecting adjustment sets that suffice to control for confounding. In a
causal graphical model, a set of covariates $Z$ is an adjustment set for the
effect of a,  possibly individualized,  point exposure treatment rule on an outcome,  if the interventional mean,  i.e. the mean of the
outcome in the hypothetical world in which all units in the population receive a
given treatment rule,  is identified by the g-formula \citep{robins1986} that
adjusts for $Z$.  An individualized point exposure treatment rule is a rule for
assigning subjects to treatment at a single time point that depends on the
subjects' covariates.  This paper contributes to the existing literature by
incorporating the possibility that variables in the graph are associated
with different costs and deriving for such scenarios a graphical criterion
for determining the optimal adjustment set among those satisfying a particular cost constraint. 

The literature on graphical selection of adjustment sets is particularly
relevant for the task of deciding, at the stage of the design of
an observational study, which variables to measure in order to control for
confounding. It is well known that, under a causal graphical model, there
may exist many adjustment sets,  and a series of recent papers  derived
complete and sound graphical criteria for determining them 
\citep{pearl-causality,
kuroki-miyawa, shpitser-adjustment, perkovic2018}. A subsequent thread of
papers provided graphical rules for comparing adjustment sets in causal graphical
models based on efficiency criteria. Specifically, assuming a linear causal
graphical model with no hidden variables and treatment effects estimated by
ordinary least squares, \cite{kuroki-cai} and \cite{kuroki-miyawa} provided
criteria for comparing certain pairs of adjustment sets. \cite{perkovic}
derived a graphical characterization of the globally most efficient
adjustment set and extended the criteria of Kuroki, Cai and Miyakawa. See
also \cite{didelez}. \cite{eff-adj} extended these results to non-parametric
graphical models and non-parametrically adjusted estimators. All of the
aforementioned papers considered only static treatment rules, i.e. `one
size fits all regimes' which assign the same treatment to all population
units regardless of their covariates.   \cite{us_bka} extended the
results of \cite{eff-adj} by allowing the possibility of both individualized
treatment rules and graphical models that have hidden variables.  The authors also
provided graphical criteria for determining optimal adjustment sets
both among  minimal adjustment sets and among minimum cardinality adjustment
sets. Minimal adjustment sets are valid adjustment sets such that the
removal of any variable from them destroys their validity. Moreover, \cite{us_bka} 
 provided a sufficient criterion for the existence of a globally
optimal adjustment set.  \cite{runge} provided a necessary and sufficient
criterion for the existence of a globally optimal adjustment set in linear
causal graphical models.  Assuming a non-graphical online setting in which the
investigator can alter the data collection mechanism adaptively, \cite{malek}
proposed an estimator of an optimal identifying
functional among those in a predefined set.

In most realistic settings the costs associated with measuring different
covariates can vary considerably. For instance, variables requiring the
laboratory assessment of blood samples are usually much more
expensive to measure than variables that can be obtained by clinical examination. The
latter, in turn, are more expensive than those obtained from surveys. Cost
considerations then give rise to the important practical problem of
determining the adjustment set that yields estimators of  treatment
effects with minimum variance among the adjustment sets whose overall cost
meets a given budget constraint. To the best of our knowledge the literature
on causal graphical models has not addressed this problem yet.  The present paper
fills this gap.  In fact, we show by means of an example, that such an ideal
adjustment set does not always exist,  except when the variance minimization problem is
restricted to observable adjustment sets with lowest overall cost.  We
derive a graphical criterion and a polynomial time algorithm for computing
a solution to the latter optimization problem.  We refer to adjustment sets that solve this problem as optimal minimum cost adjustment sets.  Our results are based on building
a special flow network associated to the original causal graph. We show that
optimal minimum cost adjustment sets can be found by computing a maximum
flow on the network, and then finding the set of vertices that are reachable
from the source  by augmenting paths. Flow networks had
already been proposed in \cite{acid} and \cite{vanAI} as a tool to compute minimal and
minimum cost adjustment sets, but with no consideration for statistical
efficiency.
The \texttt{python} package \texttt{optimaladj} available at %
\url{https://pypi.org/project/optimaladj} implements the algorithms
introduced in this paper.

The rest of the paper is organized as follows. In Section \ref{sec:back} we 
review some necessary background on graph theory and semiparametric
estimation. In Section \ref{sec:causal} we present the non-parametric causal
graphical model we will assume throughout the paper. In Section \ref%
{sec:mincost_charact} we provide a graphical characterization of minimum
cost adjustment sets as a class of vertex separators in the undirected graph
introduced in \cite{us_bka}. In Section \ref{sec:nonparam} we provide a
graphical criterion for efficiency comparisons of minimum cost adjustment
sets, based on the aforementioned undirected graph. Section \ref%
{sec:mincost_flow} contains the main results of this paper. In it, we
construct a special flow network and show how min-cuts in it are related to
minimum cost adjustment sets. We also show that an optimal minimum cost
adjustment set always exists, and provide a polynomial time algorithm to
compute it. Finally, we illustrate our results in a few examples and show
that the optimal budget constrained adjustment set discussed earlier does not exist in general.

\section{Background}

\label{sec:back}

\subsection{Undirected graphs}

\label{sec:back_undirected}

An undirected graph $\mathcal{H}=({V},{E})$ is formed by a finite vertex set $%
{V}$ and a set of undirected edges ${E}$.  
A weighted undirected graph is an undirected graph together with a cost function $c: V \to (0, +\infty)$.  The cost of a set of vertices $Z$ is defined as the sum of the costs of the vertices that comprise $Z$.  In a slight abuse of notation we write $c(Z)$ for the cost of $Z$.

If $U - W$ is an edge in $\mathcal{H}$ then we say that $U$ and $W$ are
adjacent.  A path between $U$ and $W$ is a sequence of adjacent
vertices $(V_1, \dots, V_{j})$ such that $V_{1}=U $ and $V_{j}=W$. 
Two vertices $U$ and $W$ are connected in $\mathcal{H}$ if there
exists a path from $U$ to $W$ in $\mathcal{H}$.

For $Z_{1}, Z_{2}$ and $Z_{3}$ disjoint sets of vertices in $\mathcal{H}$,
we write $Z_{1} \perp_{\mathcal{H}} Z_{2} \mid Z_{3}$ if every path in $%
\mathcal{H}$ between $Z_{1}$ and $Z_{2}$ intersects $Z_{3}$. Given two
vertices $A$ and $Y$ in $\mathcal{H}$, a set of vertices $Z$ disjoint with $A
$ and $Y$ is an $A-Y$ separator  if $A \perp_\mathcal{H} Y \mid Z$. The set ${Z} $
is a minimal $A-Y$ separator if it is an $A-Y$ separator and no proper subset of ${Z}$
is an $A-Y$ separator.  If $\mathcal{H}$ is a weighted undirected graph with weight function $c$,  the set ${Z} $ is a minimum cost $A-Y$ separator if it is an $A-Y$ separator that satisfies $c(Z) \leq c(Z^{\prime})$ for any other 
$A-Y$ separator $Z^{\prime}$.

\subsection{Directed graphs}

\label{sec:back_directed}

A directed graph $\mathcal{G}=({V},{E})$ is formed by a  finite vertex set ${V}
$ and a set of directed edges ${E} \subset {V}\times{V}$.  Given ${Z}\subset {V}$ the induced subgraph $\mathcal{G}_{{Z}}=({Z},{E}%
_{Z})$ is defined as the graph obtained by considering only vertices in ${Z}$
and edges between vertices in ${Z}$. 

Two vertices $U$ and $W$ are adjacent if there is an edge between them. A
path between $U$ and $W$ in $\mathcal{G} $ is a sequence of adjacent
vertices $(V_1, \dots, V_{j})$ such that $V_{1}=U$ and $V_{j}=W$. The path
is directed (or causal) if $V_{i}\to V_{i+1}$ for all $i \in \lbrace
1,\dots,j-1\rbrace$.

If $U\rightarrow W$, then $U$ is a parent of $W$. If there is a directed
path from $U$ to $W$, then $U$ is an ancestor of $W$ and $W$ a descendant of 
$U$. We follow the convention that every vertex is an ancestor and a
descendant of itself. The sets of parents, ancestors and descendants of $W$
in $\mathcal{G}$ are denoted by $\pa_{\mathcal{G}}(W)$, $\an_{\mathcal{G}}(W)
$ and $\de_{\mathcal{G}}(W)$ respectively. The set of non-descendants of $W$
is defined as $\nd_{\mathcal{G}}(W)\equiv {V}\setminus \de_{\mathcal{G}}(W) $%
. For a set of vertices ${Z} $ we define $\an_{\mathcal{G}}({Z})=\cup_{W\in{Z%
}} \an_{\mathcal{G}}(W) $ and $\de_{\mathcal{G}}({Z})=\cup_{W\in{Z}} \de_{%
\mathcal{G}}(W) $.


A directed cycle is a directed path that begins and ends at the same vertex.  A directed acyclic graph is a directed graph that does not have 
directed cycles.

Let $\mathcal{G}$ be a directed acyclic graph and let $A$ and $Y$ be
vertices in $\mathcal{G}$. 
Let $\cn(A,Y,\mathcal{G})$ be the set
of vertices that lie on a directed path between $A$ and $Y$ and are not
equal to $A$.  
Let $\forb(A,Y,\mathcal{G})\equiv \de_{\mathcal{G}%
}\left( \cn(A,Y,\mathcal{G})\right) \cup \left\{ A\right\} $.  We call this
the set of forbidden vertices with respect to $A$, $Y$ in $\mathcal{G}$.

The proper back-door graph $\mathcal{G}^{pbd}(A,Y)$ %
\citep{vanAI} is defined as the graph formed by removing from $\mathcal{G}$
the first edge of every directed path from $A$ to $Y$.
%

The moral graph $
\mathcal{G}^{m}$ associated with  a directed acyclic graph $\mathcal{G}$,  is an undirected graph with the same vertex set as $
\mathcal{G}$ and an edge $U- W$ if any of the following hold in $\mathcal{G}$%
: $U\rightarrow W$, $W\rightarrow U$, or there exists a vertex $C$ such that 
$U \rightarrow C \leftarrow W$.

\subsection{Flow networks}
\label{sec:back_flow}
We follow the conventions in \cite{even}.  A flow network $\mathcal{D}=(V,  E,  k,  s,  t)$ is a directed graph together with a capacity function $k: E \to [0,  +\infty]$ and two distinguished vertices $s$ and $t$.  $s$ is called the source and $t$ the sink of the network.  $k(e)$ is called the capacity of edge $e$. Consider then a flow network $\mathcal{D}=(V,  E,  k,  s,  t)$.  For a vertex $W \in V$ we let $\alpha(W)$ be the set of edges that point into $W$,  and $\beta(W)$ be the set of edges that point out of $W$.  For a set of vertices $S \subset V$ we let $\overline{S}=V \setminus S$

A flow is a function $f:E \to \mathbb{R}$ that satisfies $0 \leq f(e) \leq k(e)$ for all $e\in E$ and 
$
\sum_{e \in \alpha(W)} f(e) = \sum_{e \in \beta(W)} f(e)
$
for every vertex $W$ not equal to the source or the sink.  The total flow of $f$ is defined as 
$
\sum_{e \in \alpha(t)} f(e).
$
A flow is called a max-flow if no other flow has a greater total flow.

A cut is a set of vertices that contains the source but not the sink of the network.  If $S$ is a set of vertices we define $(S, \overline{S})$ as the set of edges in $\mathcal{D}$ of the form $U \to W$ for some $U\in S$ and $W \notin S$. 
If $S$ is a cut we define its capacity as
$$
k(S) \equiv \sum\limits_{e\in (S, \overline{S})} k(e).
$$
A cut is called a min-cut if no other cut has a smaller capacity.

\subsection{Semiparametric estimation}
\label{sec:back_semiparam}

An estimator $\widehat{\gamma }$ of a parameter 
$\gamma \left( P\right) $ based on $n$ independent identically distributed
random copies ${V}_{1},\dots ,{V}_{n}$ of ${V}$ is 
asymptotically linear at a probability law $P$ if there exists a random variable $\varphi
_{P}\left( {V}\right) $, called the influence function of $\gamma (P)$,  such that $E_{P}\left\lbrace \varphi _{P}\left( {V}\right) \right\rbrace=0$%
, $var_{P}\left\lbrace \varphi _{P}\left( {V}\right) \right\rbrace<\infty$
and $n^{1/2}\left\{ \widehat{\gamma }-\gamma \left( P\right) \right\}
=n^{-1/2}\sum_{i=1}^{n}\varphi _{P}\left( {V}_{i}\right) +o_{p}(1)$ under $P$%
.  The Central Limit Theorem implies that if  $%
\widehat{\gamma }$ is asymptotically linear,   $n^{1/2}\left\{ \widehat{\gamma }-\gamma
\left( P\right) \right\} $ converges in distribution to a zero mean normal
distribution with variance $var_{P}\left\{ \varphi _{P}\left( {V}_{i}\right)
\right\} $. Given a collection of probability laws $\mathcal{P}$ for ${V}$,
an estimator $\widehat{\gamma }$ of $\gamma \left( P\right) $ is said to be
regular at one $P$ if its convergence to $\gamma \left( P\right) $ is
locally uniform at $P$ in $\mathcal{P}$ \citep{van1998asymptotic}.

\section{Causal graphical models}
\label{sec:causal}

Given a directed acyclic graph $\mathcal{G}$ with vertex set $V$,  we identify $V$ with a random vector. 
The Bayesian Network $\mathcal{M}\left( \mathcal{G}\right) $ is the collection of laws $P$ for $V$ that satisfy the Local Markov Property: $$W\perp \!\!\!\perp \nd_{\mathcal{G}%
}\left( W\right) \text{ }|\text{ }\pa_{\mathcal{G}}\left( W\right) \text{
under }P \text{ for all }W\in {V}.$$  Here $A\perp \!\!\!\perp B|C$ stands for
conditional independence of $A$ and $B$ given $C.$  
Throughout,  we will assume that the law $P$ of $V$ admits a density $f$ with respect to some dominating
measure.  Then the Local Markov Property implies \citep{pearl-causality} 
\begin{equation}
f\left( {v}\right) =\prod\limits_{V_{j}\in {V}}f\left\{ v_{j}\mid \pa_{%
\mathcal{G}}(v_{j})\right\} ,  \label{eq-factorization}
\end{equation}%
where $\pa_{\mathcal{G}}(v_{j})$ is the value taken by $\pa_{\mathcal{G}%
}\left( V_{j}\right)$ when ${V}$ takes the value ${v}$.

In this paper we will assume an agnostic causal  graphical model %
\citep{spirtes, mediation} represented by a directed acyclic graph $\mathcal{%
G}$.  This  model identifies the vertex set of $\mathcal{G}$ with a factual
random vector ${V }$ and assumes that: (i) the law 
$P$ of ${V}$ satisfies $P\in \mathcal{M}\left( \mathcal{G%
}\right) $ and (ii) for any $A\in {V}$, ${L}\subset \nd_{\mathcal{G}}(A)$
and $\pi(A\mid {L})$ a conditional law for $A$ given ${L}$, the intervention
density $f_{\pi}\left( {v}\right) $ of the variables in $\mathcal{G}$ when,
possibly contrary to fact, the value of $A$ is drawn from the law $\pi(A\mid 
{L})$ is given by 
\begin{equation}
f_{\pi}\left( {v}\right) =\pi(a\mid {l})\prod\limits_{V_{j}\in {V}\setminus
\{A\}}f\left\{ v_{j}\mid \pa_{\mathcal{G}}(v_{j})\right\} ,
\label{eq:g-form}
\end{equation}%
where $a$ and ${l}$ are the values taken by $A$ and ${L}$ when ${V}$ takes the value ${v}$.  Formula \eqref{eq:g-form} is known as the g-formula %
\citep{robins1986}.  
 The conditional law 
$\pi$ designates a,  possibly random and individualized,  treatment rule.  A non-random
individualized treatment rule that sets $A=d\left( {L}\right) $ 
corresponds to the point mass conditional law $\pi(a\mid {l})=I_{d({l})}(a).$
In particular, a constant function $d({L})=a$ corresponds to a static
intervention that sets $A=a$.
Throughout the paper we let $Y$ and $A$ be the outcome and treatment of interest respectively.  Let 
$\chi _{\pi }(P;\mathcal{G})$ be the mean of $Y$ under $f_{\pi }$.  We refer
to $\chi _{\pi }(P;\mathcal{G})$ as the inverventional mean under treatment rule $\pi
$.  
By the factorizations $\left( \ref{eq-factorization}\right) $ and $\left( %
\ref{eq:g-form}\right)$,  the Radon-Nykodim theorem gives 
\begin{equation*}
\chi _{\pi }(P;\mathcal{G})=E_{P}\left[ \frac{\pi \left( A\mid {L}\right) }{%
f\left\{ A\mid \pa_{\mathcal{G}}(A)\right\} }Y\right] . 
\end{equation*}
Furthermore, the Local Markov property implies that 
$$\chi _{\pi }(P;\mathcal{%
G}) =E_{P}\left( E_{\pi ^{\ast }}\left[ E_{P}\left\{ Y\mid A,\pa_{\mathcal{G}%
}(A),{L}\right\} \mid \pa_{\mathcal{G}}(A),{L}\right] \right),
$$ 
where $E_{P}\left( \cdot |\cdot \right) $ stands for the conditional mean under $%
P$ and $E_{\pi ^{\ast }}\left( \cdot |\cdot \right) $ stands for the conditional
mean under the conditional law of $A$ given ${L}$ and $\pa_{\mathcal{G}}(A)$
defined as $\pi ^{\ast }\{A\mid \pa_{\mathcal{G}}(A),{L}\}\equiv \pi (A\mid {%
L})$.

We are interested in conducting inference about $\chi _{\pi}(P;\mathcal{G})$ when only
a subset ${N}$ of ${V}$ is observable. The inferential problem is thus
defined by the following assumptions: (i)  $P\in \mathcal{M}\left( \mathcal{G}\right) ,$ (ii) the available data
consists of a random sample from the marginal law of ${N}$ under $P$, (iii)
the parameter of interest is $\chi _{\pi}(P;\mathcal{G})$ and (iv) at least
one observable adjustment set exists.  Adjustment sets will be formally defined in the following section. 
Throughout we will
assume (i) $Y \in \de_{\mathcal{G}}(A)$, (ii)$\left\{ A,Y\right\} \cup {L}%
\subset {N}$,  (iii) ${L}\subset \nd_{\mathcal{G}}\left( A\right) $, and (iv) $A$ takes values in a finite set .
Moreover,  we assume that each observable variable $W\in N$ has an associated positive cost $c(W)$,  that could represent,  for example,  the cost of measuring $W$.

\section{Minimum cost adjustment sets and their graphical characterization}\label{sec:mincost_charact}

\cite{us_bka} gave the following definitions of dynamic  adjustment sets and minimal dynamic adjustment sets in graphs with hidden variables.

\begin{definition}
A set ${Z\subset
V\backslash }\left\{ A,Y\right\} $ is an ${L}-{N}$ dynamic adjustment set
with respect to $A,Y$ in $\mathcal{G}$ if ${L}\subset {Z} \subset N$ and
for all conditional laws $\pi (A\mid {L})$ for $A$ given ${L}$%
, all $P\in \mathcal{M}\left( \mathcal{G}\right) $ and all $y\in \mathbb{R}$%
\begin{align}
&E_{P}\left( E_{\pi ^{\ast }}\left[ E_{P}\left\{ I_{(-\infty ,y]}(Y)\mid A,%
\pa_{\mathcal{G}}(A),{L}\right\} \mid \pa_{\mathcal{G}}(A),{L}%
\right] \right) =  \notag \\
&E_{P}\left( E_{\pi _{{Z}}^{\ast }}\left[ E_{P}\left\{ I_{(-\infty
,y]}(Y)\mid A,{Z}\right\} \mid {Z}\right] \right) ,
\nonumber
\end{align}%
where $\pi _{{Z}}^{\ast }(A\mid {Z})\equiv \pi (A\mid {L%
})$ and, recall, $\pi ^{\ast }\left\{ A\mid \pa_{\mathcal{G}}(A),{L}%
\right\} \equiv \pi (A\mid {L})$.

An  ${L}-{N}$  dynamic adjustment set $Z$ is minimal if no proper subset of $Z$ is  an ${L}-{N}$  dynamic adjustment set .
\end{definition}

This extends the definition of \cite{shpitser-adjustment} and \cite{maathuis2015} to accommodate,  possibly random,  $L$ dependent treatment rules and graphs with hidden variables.  
\cite{us_bka} also provided a characterization of minimal $L-N$  dynamic adjustment sets as minimal $A-Y$ separators in a suitably constructed undirected graph.  We will review this characterization in Section \ref{sec:graph_charact_undirected}.  For conciseness,  in what follows we drop the dynamic apellative and simply write $L-N$ adjustment sets.  Also,  all  $L-N$ adjustment sets are with respect to $A,Y$ in $\mathcal{G}$.

We define the cost of an $L-N$ adjustment set $Z$ as $\sum_{W\in Z} c(W)$,  and in a slight abuse of notation we denote this cost with $c(Z)$.
\begin{definition}
A set $Z$ is a minimum cost ${L}-{N}$  adjustment set if it is an $L-N$ adjustment set that satisfies
$
c(Z) \leq c(Z^{\prime})
$
for all ${L}-{N}$  adjustment sets $Z^{\prime}$.
\end{definition}

Consider the design of a study aimed at estimating the interventional mean under an $L$ dependent treatment rule $\pi$.    Suppose that the investigator has postulated a causal graphical model to this end,  and that due to practical or ethical reasons,  she can only observe a subset ${N}$ of the variables in $\mathcal{G}$.
Suppose further that $\mathcal{G}$ includes at least one $L-N$  adjustment set $Z$.  This implies that $A, Y$ and $Z$
suffice to 
identify the interventional mean $\chi _{\pi }(P;\mathcal{G})$ with the so called g-functional  \citep{robins1986}
\begin{equation*}
\chi _{\pi ,{Z}}(P;\mathcal{G}) \equiv E_{P}\left[ E_{\pi _{{Z}}^{\ast
}}\left\{ E_{P}\left( Y\mid A,{Z}\right) \mid {Z}\right\} \right]  
\end{equation*}
which can then be estimated non-parametrically as further explained in Section \ref{sec:nonparam}.  
For economic reasons, the investigator may then choose to use a minimum cost $L-N$ adjustment set.  If several minimum cost $L-N$  adjustment sets exist then,  as we further discuss in Section \ref{sec:eff_comp},  a reasonable criterion for comparing them is using the variance of the limiting distribution of the resulting non-parametric estimators of the g-functional.  

Our goals in this paper are to:
\begin{enumerate}
\item Provide a graphical characterization of minimum cost $L-N$  adjustment sets.
\item Prove the existence of an optimal minimum cost $L-N$  adjustment set that yields non-parametric estimators of the interventional mean with the
smallest asymptotic variance  among those that
control for minimum cost $L-N$  adjustments and provide a polynomial time graphical algorithm to compute it.
\end{enumerate}
To achieve this,  we will leverage the undirected graph defined in \cite{us_bka},  which we review next.  In what follows,  we will assume that there exists at least one $L-N$ adjustment set in $\mathcal{G}$.

\subsection{Minimum cost adjustment sets and undirected graphs}
\label{sec:graph_charact_undirected}

Let  $\mathcal{H}^{0}\equiv \left\{ \mathcal{G}_{\an_{%
\mathcal{G}}(\{A,Y\}\cup {L})}^{pbd}(A,Y)\right\} ^{m}$ and
$\ignore\equiv \left\{ \an_{\mathcal{G}}(\{A,Y\}\cup 
{L})\setminus \{A,Y\}\right\} \cap \left\{ \left[ V \setminus N \right] \cup \forb(A,Y,\mathcal{G}%
)\right\} . 
$
Thus,  $\ignore$ is the subset of the vertices $\mathcal{H}^{0}$ that are not equal to $A$ or $Y$ and that are either not observable ($  V \setminus N $) or are variables that cannot be members of
any $L-N$ adjustment set ($\forb(A,Y,\mathcal{G}) $,  see \cite{shpitser-adjustment}).

\begin{definition}
\label{def:H} 
The non-parametric adjustment efficiency graph associated with $(A,Y,{L%
},{N})$ in $\mathcal{G}$, denoted with $\mathcal{H}^{1}$ 
is the undirected graph constructed
from $\mathcal{H}^{0}$ by (1) removing all
vertices in $\ignore$, (2) adding an
edge between any pair of remaining vertices if they were connected in $
\mathcal{H}^{0}$ by a path with vertices in $%
\ignore$ and (3) adding an edge between $A$ and each vertex in ${L}$ and between $Y$ and each
vertex in ${L}$.
\end{definition}

In words,  $\mathcal{H}^{1}$ is obtained from $\mathcal{H}^{0}$ by first
performing a latent projection on $V \setminus \ignore$ and then connecting all vertices in $L$ to both $A$ and $Y$. 
Similar constructions were also used in \cite{textor12}
and \cite{van14}.  Note that even though $\mathcal{H}^{0}$ and $\mathcal{H}^{1}$ depend on $A, Y,  L,  N$ and $\mathcal{G}$,  for brevity we omit this dependence in the notation.

Proposition 2 of \cite{us_bka} states that ${Z}$ is a minimal ${L}-{N}$   adjustment
set if and only if $%
{Z}$ is a minimal $A-Y$ separator in $\mathcal{H}^{1}$.  Note that  \cite{us_bka} uses the term cut to refer to what we here call a separator.  Both terms are used in the literature,  and in this paper we prefer to reserve the name cut for the concept in flow networks.  Now,  since all minimum cost  ${L}-{N}$ adjustment sets are minimal ${L}-{N}$ adjustment sets,  Proposition 2 of \cite{us_bka} implies the following.

\begin{lemma}\label{lemma:characterize_mincot}
${Z}$ is a minimum cost ${L}-{N}$  adjustment
set if and only if $%
{Z}$ is a minimum cost $A-Y$ separator in $\mathcal{H}^{1}$
\end{lemma}
In what follows,  for brevity,  all separators in $\mathcal{H}^{1}$ are between $A$ and $Y$.  In the next section we review the aspects of the theory of non-parametric estimation of the g-functional $\chi _{\pi ,{Z}}(P;%
\mathcal{G})$ that are relevant to our derivation of the optimal minimum cost $L-N$ adjustment set.

\section{Non-parametric estimation of the g-functional}
\label{sec:nonparam}

 \cite{AIDS} showed that estimators of $\chi _{\pi ,{Z}}(P;%
\mathcal{G})$ that are regular and asymptotically linear at all $P$ in a model $\mathcal{P}$ that only makes assumptions
on the complexity or smoothness $b\left( A,{Z};P\right) \equiv E_{P}\left( Y|A,{Z}\right) $ and/or $%
f\left( A\mid {Z}\right)$ have a unique influence function $\psi _{P,\pi
}\left({Z};\mathcal{G}\right) $ given by 
\begin{equation}
\psi _{P,\pi }\left( {Z};\mathcal{G}\right) \equiv \frac{\pi (A\mid {L})}{%
f\left( A\mid {Z}\right) }\left\{ Y-b\left( A,{Z};P\right) \right\} +E_{\pi
_{{Z}}^{\ast }}\left\{ b\left( A,{Z};P\right) \mid {Z}\right\} -\chi
_{\pi,Z}\left( P,\mathcal{G}\right) .
\nonumber
\end{equation}%
Note that even though $\psi _{P,\pi }$ is also a function of $A$ and $Y$,  this is not reflected in the notation,  for the sake of brevity.

There exist multiple estimation strategies that that rely on making smoothness or complexity type assumptions on $b\left( A,{Z};P\right)$ and/or $%
f\left( A\mid {Z}\right)$.  We list a few of them next.  The
inverse probability weighted estimator is given by $\widehat{\chi }_{\pi ,IPW}=\mathbb{P}%
_{n}\left\{ \widehat{f}\left( A\mid {Z}\right) ^{-1}\pi (A\mid {L})Y\right\}
,$ where $\widehat{f}\left( A|Z\right) $ is a non-parametric estimator of $f\left( A\mid {Z}\right) $
\citep{hirano}.  The outcome regression estimator is given by $\mathbb{P}_{n}\left[
E_{\pi _{{Z}}^{\ast }}\left\{ \widehat{b}\left( A,{Z}\right) \mid {Z}%
\right\} \right] $ where $\widehat{b}$ is a non-parametric
estimator of $b$ \citep{hahn}.  The doubly-robust estimator %
\citep{vanderlaan, chernozhukov2018double, smucler}, also known as augmented IPW,  uses non-parametric estimators of both $f\left( A\mid 
{Z}\right) $ and $b\left( A,{Z}\right) $.  
Examples of non-parametric estimators of $f\left( A\mid 
{Z}\right) $ and $b\left( A,{Z}\right) $ include series or kernel based estimators,  estimators based on boosted trees and other machine learning techniques.

We will refer to
estimators that are regular and asymptotically linear with unique influence
function $\psi _{P,\pi }\left({Z};\mathcal{G}\right) $ as non-parametric estimators that adjust for $Z$. 
It follows from the
discussion above that if $\widehat{\chi }_{\pi ,{Z}}$ is a  non-parametric estimator that adjusts for $Z$, then
 ${n}^{1/2}\left\{ \widehat{\chi }_{\pi ,{Z}}-\chi _{\pi }\left(
P;\mathcal{G}\right) \right\} $ converges in distribution to $N\left\{
0,\sigma _{\pi ,{Z}}^{2}\left( P\right) \right\} $ where $\sigma _{\pi ,{Z}%
}^{2}\left( P\right) \equiv var_{P}\left\{ \psi _{P,\pi }\left( {Z};\mathcal{%
G}\right) \right\} .$

\subsection{Efficiency comparison of minimum cost adjustment sets}
\label{sec:eff_comp}

Define the following preorder on the class of $L-N$ adjustment sets:
$$
Z_1 \preceq_{{L}} Z_2  \Longleftrightarrow \sigma _{\pi ,{Z_1}}^{2}\left( P\right) \leq \sigma _{\pi ,{Z}_2}^{2}\left( P\right) \text{ for all }\pi(A\mid L) \text{ and all } P\in\mathcal{M(G)}.
$$
In words, $Z_1 \preceq_{{L}} Z_2$ if adjusting for $Z_1$ yields more efficient non-parametric estimators of  $\chi _{\pi }(P;\mathcal{G})$ than adjusting for $Z_2$,  uniformly over all possible treatment rules $\pi$ and laws $P$ in the Bayesian Network $\mathcal{M(G)}$.
Define the following relation between separators in $\mathcal{H}^{1}$:
$$
Z_1 \unlhd_{\mathcal{H}^{1}} Z_2 \Longleftrightarrow  Y\perp_{\mathcal{H}^{1}} Z_2 \setminus Z_1 \mid Z_1 \text{ and } A \perp_{\mathcal{H}^{1}} Z_1 \setminus Z_2 \mid Z_2.
$$
Since,  as stated in Lemma \ref{lemma:characterize_mincot},  minimum 
cost $L-N$ adjustment sets and minimum cost separators in $\mathcal{H}^{1}$ are equivalent,  using Lemma \ref{lemma:characterize_mincot} and Propositions 3 and 5 of \cite{us_bka},  we can deduce the following graphical criterion for comparing minimum cost $L-N$ adjustment sets.
\begin{lemma}\label{lemma_order_implies_var}
Let $Z_1$ and $Z_2$ be minimum cost $L-N$ adjustment sets.  Then
\begin{equation}
Z_1 \unlhd_{\mathcal{H}^{1}} Z_2  \Longrightarrow Z_1 \preceq_{{L}} Z_2.
\label{eq:order_implies_var}
\end{equation}
\end{lemma}

We now argue why asymptotic efficiency is a reasonable basis for comparing minimum cost adjustment sets,  in the context of designing a planned study where the cost associated with each observable variable in the graph reflects the cost of measuring the variable  on one subject. 
Consider two minimum cost adjustment sets $Z_1$ and $Z_2$, and let $\widehat{\chi }_{\pi ,{Z_1}}$ and $\widehat{\chi }_{\pi ,{Z}_2}$ be non-parametric estimators that adjust for $Z_1$ and $Z_2$ respectively.   Suppose we know that $Z_1 \preceq_{{L}} Z_2$.  If we want the length of the  95\% Wald confidence interval for  $\chi _{\pi }(P;\mathcal{G})$ to be bounded by $M$ then using $\widehat{\chi }_{\pi ,{Z_1}}$ 
we will need  $n_1 \approx \left\lbrace 3.92 \: \sigma _{\pi ,{Z}_1%
}\left( P\right) M^{-1} \right\rbrace^{2}$ samples,  whereas using $\widehat{\chi }_{\pi ,{Z_2}}$ we will need $n_2 \approx \left\lbrace 3.92 \: \sigma _{\pi ,{Z}_2%
}\left( P\right) M^{-1}\right\rbrace^{2}$ samples.  Since $\sigma _{\pi ,{Z}_1%
}^{2}\left( P\right) \leq  \sigma _{\pi ,{Z}_2
}^{2}\left( P\right)$ we have that $n_1 \leq n_2$ and hence 
$n_1\times c(Z_1) \leq n_2 \times c(Z_2)$.  In words, for the same level of precision,  the total cost of using $Z_1$ as an adjustment set will be lower than that of using  $Z_2$.

In the following section we show that there exists a minimum cost $L-N$ adjustment set,  which we denote $O_{c}$,  that satisfies that for any other minimum cost $L-N$ adjustment set $Z$ it holds that $O_{c} \preceq_{L} Z$.  We call  $O_{c}$ an optimal minimum cost $L-N$  adjustment set. To show this,  we will make a connection between minimum cost $L-N$ adjustment sets and min-cuts in a suitably constructed flow network.  This construction will also allow us to derive a polynomial time algorithm to compute $O_{c}$.

\section{Optimal minimum cost adjustment sets and network flows}
\label{sec:mincost_flow}

The following network flow construction is inspired by the construction in Theorem 6.4 of \cite{even}.  The main difference is that \cite{even} puts unit capacity on all `internal edges'.

\begin{definition}\label{def:network}
Let the flow network $\mathcal{D}$ be defined as follows.  For each vertex $W$ in $\mathcal{H}^{1}$ add two vertices $W^{\prime}$ and $W^{\prime\prime}$ and the edge $W^{\prime} \rightarrow W^{\prime\prime}$ to $\mathcal{D}$.  We call these internal edges.  If there is an edge joining $U$ and $W$ in $\mathcal{H}^{1}$ add edges $U^{\prime\prime} \rightarrow W^{\prime}$ and $W^{\prime\prime} \rightarrow U^{\prime}$.  We call these external edges.  Thus an edge $U - W$ in $\mathcal{H}^{1}$ gives place to the following structure in $\mathcal{D}$:
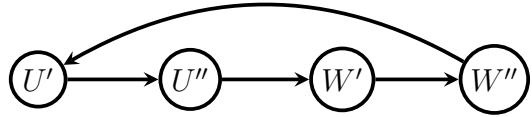
\begin{center}
\begin{tikzpicture}[>=stealth, node distance=2cm,
pre/.style={->,>=stealth,ultra thick,line width = 1.4pt}]
 \begin{scope}
    \tikzstyle{format} = [circle, inner sep=2pt,draw, thick, circle, line width=1.4pt, minimum size=3mm]
    
\node[format] (Up) {$U^{\prime}$};
\node[format, right of=Up] (Upp) {$U^{\prime\prime}$};
\node[format, right of= Upp] (Wp) {$W^{\prime}$};
\node[format, right of=Wp] (Wpp) {$W^{\prime\prime}$};

\draw (Up) edge[pre, black] node[above] {}(Upp)  ;
\draw (Upp) edge[pre, black] node[above] {}(Wp)  ;
\draw (Wp) edge[pre, black] node[above] {}(Wpp)  ;
\draw (Wpp) edge[pre, black,out=150, in =30] node[above] {}(Up)  ;

 \end{scope} 
 \end{tikzpicture} 
\end{center}
The capacity of an internal edge $e=W^{\prime}\to W^{\prime\prime}$ is equal to the cost of $W$, that is,  $k(e)=c(W)$,  except if $W$ is equal to $A$ or to $Y$,  in which case the capacity is infinity.  The capacity of external edges is infinity.  We set $Y^{\prime\prime}$ as the source and $A^{\prime}$ as the sink of the network.
\end{definition}

We will provide a full example of this construction shortly in Figure \ref{fig:first_flow}. 
We show in Lemma \ref{lemma:finite_cap} in the Appendix that there always exists a cut in $\mathcal{D}$ with finite capacity.
Next,  we define two mappings,  $\mathfrak{d}$ and $\mathfrak{h}$.  The former maps minimal separators in $\mathcal{H}^{1}$ to sets of vertices in $\mathcal{D}$, while the latter maps cuts with finite capacity in $\mathcal{D}$ to sets of vertices in $\mathcal{H}^{1}$.
\begin{definition}\label{def:equiv_cut_sep} 
Given $Z$ a minimal separator in $\mathcal{H}^{1}$ we let
$
\mathfrak{d}(Z)
$
be the set formed by $Y^{\prime\prime}$ and all vertices in $\mathcal{D}$ that lie on some directed path $\delta$ from $Y^{\prime\prime}$ to a vertex $W^{\prime}$ for some $W\in Z$,  where $\delta$  does not intersect any other $U^{\prime}$ or $U^{\prime\prime}$ for $U\in Z$. 
Given  $S$ a cut in $\mathcal{D}$ with finite capacity we let
$$
\mathfrak{h}(S)= \lbrace W: (W^{\prime},W^{\prime\prime}) \in (S, \overline{S}) \rbrace.
$$

\end{definition}

The following proposition establishes that $\mathfrak{d}$ maps  minimum cost separators in $\mathcal{H}^{1}$ to min-cuts in $\mathcal{D}$,  and $\mathfrak{h}$ maps min-cuts in $\mathcal{D}$ to 
minimum cost separators in $\mathcal{H}^{1}$.

\begin{proposition}\label{prop:equiv_cut_sep}
$ \:$

\begin{enumerate}
\item Let $Z$ be a minimal separator in $\mathcal{H}^{1}$.    Then $\mathfrak{d}(Z)$ is a cut in $\mathcal{D}$ with $k\lbrace \mathfrak{d}(Z)\rbrace=c(Z)$.

\item Let $S$ be a cut in $\mathcal{D}$ with finite capacity.  Then $
\mathfrak{h}(S)$
is a separator in $\mathcal{H}^{1}$ with $c\lbrace \mathfrak{h}(S)\rbrace=k(S)$.
\item Let $Z$ be a minimum cost separator in $\mathcal{H}^{1}$.    Then $\mathfrak{d}(Z)$ is a min-cut in $\mathcal{D}$.

\item Let $S$ be a min-cut in $\mathcal{D}$.  Then $
\mathfrak{h}(S)$
is a minimum cost separator in $\mathcal{H}^{1}$.

\item Let $Z$ be a minimum cost separator in $\mathcal{H}^{1}$.  Then $\mathfrak{h}\left\lbrace\mathfrak{d} \left( Z \right) \right\rbrace = Z$.

\end{enumerate}

\end{proposition}

Next,  we establish a connection between the $\subset$ relation defined over min-cuts in $\mathcal{D}$ and the $\unlhd_{\mathcal{H}^{1}}$ relation defined over separators in $\mathcal{H}^{1}$.

\begin{proposition}\label{prop:subset_implies_order}
Let $S$ and $S^{\prime}$ be min-cuts such that $S \subset S^{\prime}$.  Then $\mathfrak{h}(S) \unlhd_{\mathcal{H}^{1}} \mathfrak{h}(S^{\prime})$.
\end{proposition}

Lemmas \ref{lemma:characterize_mincot} and \ref{lemma_order_implies_var}  together with Propositions \ref{prop:equiv_cut_sep} and \ref{prop:subset_implies_order} imply that  if we are able to construct a min-cut $S_{c}$ that is a subset of any other min-cut,  then $\mathfrak{h}(S_{c})$ is an optimal minimum cost $L-N$ adjustment set.  We now show how such a min-cut can be constructed.

Given a flow $f$,  we will say that a path $\delta$ connecting $Y^{\prime\prime}$ and $W$ in $\mathcal{D}$ is augmenting for $f$ if for all edges $e$ in $\delta$ oriented from  $Y^{\prime\prime}$ to $W$ it holds that $f(e) < k(e)$ and for all edges $e$ in $\delta$ oriented from  $W$ to $Y^{\prime\prime}$ it holds that $f(e) >0$.
Suppose  that we have run a maximum flow algorithm on $\mathcal{D}$,  for example the preflow push algorithm \citep{preflow_push},  and  obtained a maximum flow $f^{\ast}$.  We are now ready to define our candidate optimal minimum cost $L-N$ adjustment set.

\begin{definition}
Let $S_{c}$ be the set formed by $Y^{\prime\prime}$ and all vertices $W$ in $\mathcal{D}$ such that there exists a path from $Y^{\prime\prime}$ to $W$ that is augmenting for $f^{\ast}$.  Let $O_{c} \equiv \mathfrak{h}(S_{c})$
\end{definition}
Note that  $S_{c}$ could in principle depend on the computed max-flow $f^{\ast}$,  even if this is not made explicit in the notation.

\begin{proposition}\label{prop:Sc_smallest}
$S_c$ is a min-cut.  For any other min-cut $S$ it holds that $S_{c} \subset S$.  
\end{proposition}

The following theorem,  the main result of this paper,  establishes the optimality of $O_{c}$.  

\begin{theorem}\label{theo:main_theo}
$O_{c}$ is a minimum cost $L-N$ adjustment set.  For any other minimum cost $L-N$  adjustment set $Z$  it holds that
$
O_{c} \preceq_{L}  Z.
$
\end{theorem}

Algorithm \ref{algo:main} summarizes the steps needed to compute $O_{c}$.  The complexity of Algorithm \ref{algo:main} will depend on the sub-routine used to compute the maximum flow in the third step.   For example, when the preflow push algorithm is used,  the overall complexity of Algorithm \ref{algo:main} is bounded by $\mathcal{O}\left( \#V^{2} \sqrt{\#E} \right)$.   The fourth step of Algorithm \ref{algo:main} can be easily implemented using a small modification of the depth first search algorithm.  We provide a Python implementation of Algorithm \ref{algo:main} in the \texttt{optimaladj} package,  available on pip.  Our algorithm computes maximum flows using the implementation of the preflow push algorithm available in the \texttt{networkx} library \citep{networkx}.

\begin{algorithm}[h!]
	\SetKwInOut{Input}{input}\SetKwInOut{Output}{output}
    \SetAlgoLined\DontPrintSemicolon
    \SetKwProg{proc}{procedure}{}{}
	\proc{}
	{           
	    construct $\mathcal{H}^{1}$\\
	    construct $\mathcal{D}$\\
	    compute a maximum flow $f^{\ast}$ on $\mathcal{D}$\\
		compute $S_c$ the set of nodes reachable from $Y^{\prime\prime}$ via paths that are augmenting for $f^{\ast}$\\
		compute $\mathfrak{h}(S_c)$\\
	 \Return{$\mathfrak{h}(S_c)$}
	} 
		\caption{{\bf Pseudo-algorithm to compute $O_{c}$} }
			\label{algo:main}
\end{algorithm}

\FloatBarrier
When all variables have unit costs,  Algorithm \ref{algo:main} computes an $L-N$ adjustment set that is optimal among those of minimum cardinality.  Algorithm 1 of \cite{us_bka}  does the same thing,  but with  $\mathcal{O}\left( \#V^{3.5}  \right)$  complexity.  Thus,  Algorithm \ref{algo:main} also provides an improvement on Algorithm 1 of \cite{us_bka}  for the task of computing an optimal minimum cardinality $L-N$ adjustment set.

\subsection{Examples}\label{sec:examples}
In the following figures we illustrate the results of this section. Dashed circles designate hidden variables and rectangles
the variables that the treatment rule depends on.  The numbers below the name of each vertex in $\mathcal{G}$ and $\mathcal{H}^{1}$ represent the cost of the variable associated with the vertex.  We do not assign  costs to $A,  Y$ and variables in $\ignore$,  since their costs are not relevant for the comparison of minimum-cost $L-N$ adjustment sets.  
Figure \ref{fig:first_flow} shows the flow network $\mathcal{D}$ associated with the graphs in Figure \ref{fig:first}.  Edges with finite capacities are colored green,  with the numbers next to the edges representing capacities.  All black edges have infinite capacity.

\begin{figure}[ht!]
\centering
\subfloat[$\mathcal{G}$]{
\begin{tikzpicture}[>=stealth, node distance=1.3cm,
pre/.style={->,>=stealth,ultra thick,line width = 1.4pt}]
 \begin{scope}
    \tikzstyle{format} = [circle, inner sep=2pt,draw, thick, circle, line width=1.4pt, minimum size=3mm]
    \node[format, rectangle] (X) {$X \atop 1$};
    \node[format, below of=X] (A) {$A$};
        \node[format, below of=A] (K) {$K\atop 4$};
        \node[format, below of=K] (M) {$M$};
        \node[left of=A] (l) {};
         \node[right of=A] (r) {};
            \node[format, below of=l] (P) {$B \atop 2$};
            \node[format, below of=r] (Q) {$Q \atop 1$};
\node[format, below of=P] (R) {$R \atop 1$};
\node[format, below of=Q] (T) {$T \atop 1$};
\node[right of=R] (m) {};
\node[format, below of=m] (Y) {$Y$};
\node[format, dashed,  right of=Y] (U) {$U$};        
\node[format, right of=U] (F) {$F \atop 1$};        

\draw (X) edge[pre, black] (A);
\draw (K) edge[pre, black] (A);
\draw (P) edge[pre, black] (K);
\draw (Q) edge[pre, black] (K);
\draw (A) edge[pre, black,  out=330, in=50] (M);
\draw (P) edge[pre, black] (R);
\draw (Q) edge[pre, black] (T);
\draw (R) edge[pre, black] (Y);
\draw (T) edge[pre, black] (Y);
\draw (M) edge[pre, black] (Y);
\draw (U) edge[pre, black] (Y);
\draw (U) edge[pre, black] (F);
 \end{scope} 
 \end{tikzpicture}
 }  \qquad 
\subfloat[$\mathcal{H}^{1}$]{
\begin{tikzpicture}[>=stealth, node distance=1.3cm,
pre/.style={>=stealth,ultra thick,line width = 1.4pt}]
 \begin{scope}
    \tikzstyle{format} = [circle, inner sep=2pt,draw, thick, circle, line width=1.4pt, minimum size=3mm]
 \node[format, rectangle] (X) {$X \atop 1$};
 \node[format, below of=X] (A) {$A$};
        \node[left of=A] (l) {};
         \node[right of=A] (r) {};
                  \node[format, below of=A] (K) {$K \atop 4$};
            \node[format, below of=l] (P) {$B \atop 2$};
            \node[format, below of=r] (Q) {$Q \atop 1$};
\node[format, below of=P] (R) {$R \atop 1$};
\node[format, below of=Q] (T) {$T \atop 1$};
\node[right of=R] (m) {};
\node[format, below of=m] (Y) {$Y$};

\draw (P) edge[pre, black] (R);
\draw (P) edge[pre, black, out=45, in=135] (Q);
\draw (P) edge[pre, black] (K);
\draw (Q) edge[pre, black] (K);
\draw (K) edge[pre, black] (A);
\draw (K) edge[pre, black, out=135, in=225] (X);
\draw (Q) edge[pre, black] (T);
\draw (R) edge[pre, black] (Y);
\draw (R) edge[pre, black] (T);
\draw (T) edge[pre, black] (Y);
\draw (X) edge[pre, black] (A);
\draw (X) edge[pre, black, out=330,  in=90] (Y);
 \end{scope} 
 \end{tikzpicture}
 }
\caption{A directed acyclic graph $\mathcal{G}$ and  the undirected graph $\mathcal{H}^{1}$ associated with it.
}
\label{fig:first}
\end{figure}
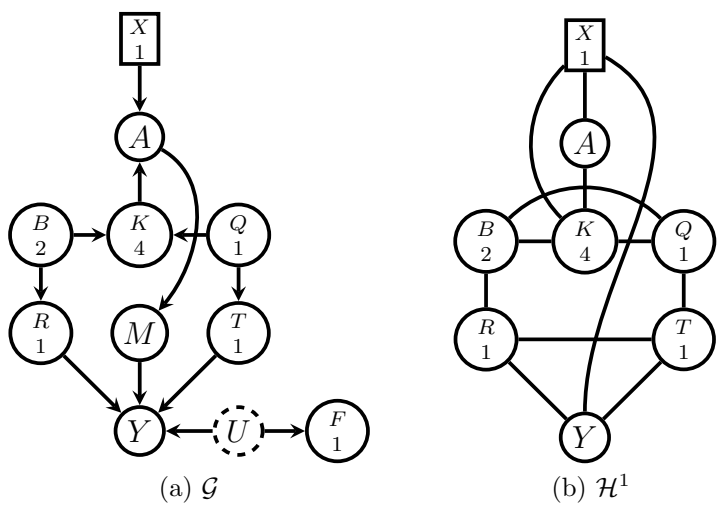

For the directed acyclic graph $\mathcal{G}$ in Figure \ref{fig:first},  let ${L}=\lbrace X \rbrace $ and ${N}=V \setminus \lbrace  U \rbrace$. Then $\an_{\mathcal{G}}(\lbrace A,Y\rbrace \cup {L})={V} \setminus \lbrace F \rbrace $, $\forb(A,Y,\mathcal{G})=\lbrace A,Y,  M\rbrace$ and $\ignore=\lbrace U, M \rbrace $.  
In $\mathcal{H}^1$,  the set of all separators is given by the collection of sets $Z$ that satisfy $X \in  Z$ and at least one of the following:
\begin{itemize}
\item $K\in Z$.
\item $B \in Z$ or $R \in Z$,  and  $Q \in Z$ or $T \in Z$.
\end{itemize}
The only minimum cost separators are $Z_1 = \lbrace X,  Q, R\rbrace$ and $Z_2 = \lbrace X,  T, R\rbrace$.  It is easy to show that $Z_2 \unlhd_{\mathcal{H}^{1}} Z_1$.  Thus,  $Z_2$ is an optimal minimum cost $L-N$ adjustment set.  
Using results from \cite{us_bka},  it is easy to show that there exists a globally optimal $L-N$ adjustment set in $\mathcal{G}$,  i.e.   an $L-N$ adjustment set that is more efficient than any other $L-N$ adjustment set,  and that it is given by $Z_2 \cup \lbrace F \rbrace$.
Turn now to the representation of $\mathcal{D}$ in Figure \ref{fig:first_flow}.  The min-cut obtained by running the preflow push algorithm and then computing $S_c$ is given by $S_c =\lbrace Y^{\prime\prime},  X^{\prime}, T^{\prime}, R^{\prime}\rbrace$.  The capacity of this cut is $k(S_c)=k\lbrace (R^{\prime},R^{\prime\prime})\rbrace+k\lbrace (T^{\prime},T^{\prime\prime})\rbrace+ k\lbrace(X^{\prime},X^{\prime\prime})\rbrace=3$.
The optimal minimum cost  $L-N$ adjustment set is then $O_{c}=\mathfrak{h}(S_c)=\lbrace X, T,R\rbrace$,  matching what we obtained earlier by analyzing separators in $\mathcal{H}^{1}$.  

\begin{figure}[ht!]
\centering
\begin{tikzpicture}[>=stealth, node distance=1.3cm,
pre/.style={->,>=stealth,ultra thick,line width = 1.4pt},  scale=0.7, every node/.style={transform shape}]
 \begin{scope}
    \tikzstyle{format} = [circle, inner sep=2pt,draw, thick, circle, line width=1.4pt, minimum size=3mm]
    
\node[format] (Xp) {$X^{\prime}$};

             \node[format,  below of=Xp] (Xpp) {$X^{\prime\prime}$};    
    \node[format,  below of=Xpp] (Ap) {$A^{\prime}$};
             \node[format,  below of=Ap] (App) {$A^{\prime\prime}$};
             \node[format,  below of=App] (Kp) {$K^{\prime}$};
             \node[format,  below of=Kp] (Kpp) {$K^{\prime\prime}$};
        \node[left of=Kpp] (l) {};
         \node[right of=Kpp] (r) {};
            \node[format, below of=l] (Pp) {$B^{\prime}$};
            \node[format, below of=r] (Qp) {$Q^{\prime}$};
\node[format, below of=Pp] (Ppp) {$B^{\prime\prime}$};
\node[format, below of=Qp] (Qpp) {$Q^{\prime\prime}$};
\node[format, below of=Ppp] (Rp) {$R^{\prime}$};
\node[format, below of=Qpp] (Tp) {$T^{\prime}$};
\node[format, below of=Rp] (Rpp) {$R^{\prime\prime}$};
\node[format, below of=Tp] (Tpp) {$T^{\prime\prime}$};

\node[right of=Rpp] (m) {};
\node[format, below of=m] (Yp) {$Y^{\prime}$};
        \node[format, below of=Yp] (Ypp) {$Y^{\prime\prime}$};

\draw (Xp) edge[pre, black!30!green] node[right] {$1$}(Xpp)  ;

\draw (Xpp) edge[pre, black] node[right] {}(Ap)  ;
\draw (Xpp) edge[pre, black,out=225, in=135] node[right] {}(Kp)  ;
\draw (Xpp) edge[pre, black,out=315, in=0] node[right] {}(Yp)  ;

\draw (Ap) edge[pre, black] node[right] {}(App)  ;
\draw (App) edge[pre, black,out=135, in=215] node[left] {}(Xp)  ;
\draw (App) edge[pre, black] node[right] {}(Kp)  ;

\draw (Kp) edge[pre, black!30!green] node[left] {$4$ }(Kpp)  ;
\draw (Kpp) edge[pre, black] node[left] {}(Pp)  ;
\draw (Kpp) edge[pre, black] node[right] {}(Qp)  ;
\draw (Kpp) edge[pre, black, out=45, in=0] node[left] {}(Ap)  ;
\draw (Kpp) edge[pre, black, out=135, in=180] node[left] {}(Xp)  ;

\draw (Pp) edge[pre, black!30!green] node[right] {$2$ }(Ppp)  ;
\draw (Ppp) edge[pre, black] node[above] {}(Qp)  ;
\draw (Ppp) edge[pre, black]  node[left] {}(Rp)  ;
\draw (Ppp) edge[pre, black,out=130, in=200]  node[left] {}(Kp)  ;

\draw (Qp) edge[pre, black!30!green] node[left] {$1$ }(Qpp)  ;
\draw (Qpp) edge[pre, black] node[below] {}(Pp)  ;
\draw (Qpp) edge[pre, black]  node[right] {}(Tp)  ;
\draw (Qpp) edge[pre, black,out=50, in=340]  node[right] {}(Kp)  ;

\draw (Rp) edge[pre, black!30!green]  node[right] {$1$ }(Rpp)  ;
\draw (Rpp) edge[pre, black] node[above] {}(Tp)  ;
\draw (Rpp) edge[pre, black] node[left] {}(Yp)  ;
\draw (Rpp) edge[pre, black,out=140, in=220] node[above] {}(Pp)  ;

\draw (Tp) edge[pre, black!30!green]  node[left] {$1$ }(Tpp)  ;
\draw (Tpp) edge[pre, black] node[right] {}(Yp)  ;
\draw (Tpp) edge[pre, black] node[below] {}(Rp)  ;
\draw (Tpp) edge[pre, black,out=40, in=320] node[above] {}(Qp)  ;

\draw (Yp) edge[pre, black] node[right] {}(Ypp)  ;
\draw (Ypp) edge[pre, black, out=140, in=215] node[left] {}(Rp)  ;
\draw (Ypp) edge[pre, black, out=40, in=325] node[right] {}(Tp)  ;
\draw (Ypp) edge[pre, black, out=135, in=180] node[right] {}(Xp)  ;

 \end{scope} 
 \end{tikzpicture} 
\caption{The flow network $\mathcal{D}$ corresponding to the directed acyclic graph in Figure \ref{fig:first}.  All black edges have infinite capacity.
}
\label{fig:first_flow}
\end{figure}
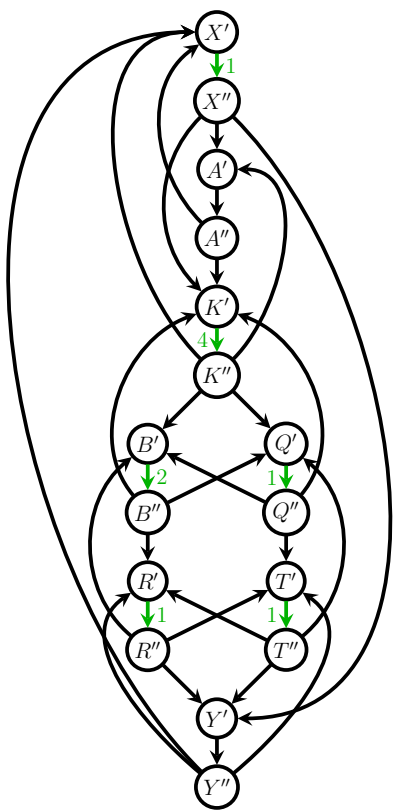

It follows from Theorem 1 of \cite{us_bka} that $O_c$ coincides with the optimal $L-N$ adjustment set among minimal $L-N$ adjustment sets.  However, it is not always the case that the optimal minimum cost and the optimal minimal $L-N$ adjustment sets are equal.  For this same graph,  if the cost $B$ were 1 and the cost of $R$ were 2 then the optimal minimal $L-N$ adjustment set would still be equal to $\lbrace X,  R, T \rbrace$,  whereas the optimal minimum cost  $L-N$ adjustment set would be $\lbrace X, B,T \rbrace$.  Since $\lbrace X,  R, T \rbrace \unlhd_{\mathcal{H}^{1}} \lbrace X, B,T \rbrace$,  this is an example in which the optimal minimal $L-N$ adjustment set is more efficient that the optimal minimum cost $L-N$ adjustment set.  The converse can never happen,  because all minimum cost $L-N$ adjustment sets are minimal $L-N$ adjustment sets.

Going back to our original example,  note that $Z_3=\lbrace X,  K\rbrace$ is the $L-N$ adjustment set with minimum possible cardinality.  It is easy to check that $Z_2 \unlhd_{\mathcal{H}^{1}} Z_3$ and thus in this case,  the optimal minimum cost $L-N$ adjustment set is more efficient than the optimal minimum cardinality $L-N$ adjustment set.  The graph in Figure \ref{fig:Omin_better} provides one example in which the reverse situation holds.

Indeed,  for the graph $\mathcal{G}$ in Figure \ref{fig:Omin_better},  let ${L}=\emptyset $ and ${N}=V $. Then $\an_{\mathcal{G}}(\lbrace A,Y\rbrace \cup {L})={V}$, $\forb(A,Y,\mathcal{G})=\lbrace A,Y\rbrace$ and $\ignore=\emptyset$.  It is easy to check that there is only one minimum cost separator in $\mathcal{H}^{1}$,  given by $Z_1=\lbrace B, Q\rbrace$.  However,   $Z_2=\lbrace T, R\rbrace$ is a minimum cardinality separator that satisfies $Z_2 \unlhd_{\mathcal{H}^{}1} Z_1$.  Thus,  in this case,  the optimal minimum cardinality $L-N$ adjustment set is more efficient that the optimal minimum cost $L-N$ adjustment set.

This example also illustrates the point made in the introduction that in general there does not exist an optimal $L-N$ adjustment set among those that satisfy an upper bound on their cost.  For the graph in Figure \ref{fig:Omin_better},  if the available budget is equal to 3,  the investigator has to choose between $\lbrace B,  Q \rbrace $,  $\lbrace B,  R \rbrace$ and $\lbrace T, Q \rbrace$.  Clearly $\lbrace B,  R \rbrace  \unlhd_{\mathcal{H}^{1}}  \lbrace B,  Q \rbrace $ and $\lbrace T,  Q \rbrace  \unlhd_{\mathcal{H}^{1}}  \lbrace B,  Q \rbrace $ and so by Propositions 3 and 5 of \cite{us_bka},  $\lbrace B,  R \rbrace  \preceq_{L}  \lbrace B,  Q \rbrace $ and $\lbrace T,  Q \rbrace  \preceq_{L}  \lbrace B,  Q \rbrace $. Thus,  the investigator actually needs to choose between  $\lbrace B,  R \rbrace$ and $\lbrace T, Q \rbrace$.  However,  in their Example 2,  \cite{eff-adj} show it is not possible to compare the asymptotic variances of these two adjustments based solely on the causal graph,  because there exist probability laws in the Bayesian Network $\mathcal{M(G)}$ under which $\lbrace B,  R \rbrace$ is more efficient but also probability laws in the Bayesian Network $\mathcal{M(G)}$ under which $\lbrace Q,  T \rbrace$ is more efficient.

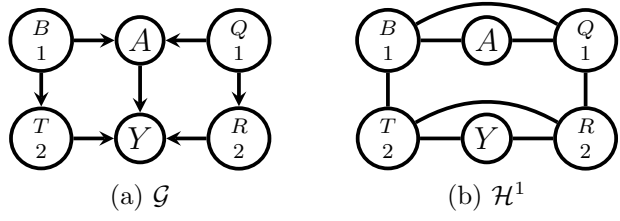
\begin{figure}[ht!]
\centering
\subfloat[$\mathcal{G}$]{
\begin{tikzpicture}[>=stealth, node distance=1.3cm,
pre/.style={->,>=stealth,ultra thick,line width = 1.4pt}]
 \begin{scope}
    \tikzstyle{format} = [circle, inner sep=2pt,draw, thick, circle, line width=1.4pt, minimum size=3mm]
            \node[format] (A) {$A$};
        \node[format, left of=A] (B) {$B \atop 1$};
\node[format, right of=A] (Q) {$Q \atop 1$};
\node[format, below of=A] (Y) {$Y$};
\node[format, below of=B] (T) {$T \atop 2$};
\node[format, below of=Q] (R) {$R \atop 2$};                 

                 \draw (B) edge[pre, black] (A);
\draw (Q) edge[pre, black] (A);

                 \draw (B) edge[pre, black] (T);
                 \draw (Q) edge[pre, black] (R);                 
                 \draw (T) edge[pre, black] (Y);                                  
 \draw (R) edge[pre, black] (Y);              
  \draw (A) edge[pre, black] (Y);                                  
                    
 \end{scope} 
 \end{tikzpicture}

}\qquad 
\subfloat[$\mathcal{H}^{1}$]{
\begin{tikzpicture}[>=stealth, node distance=1.3cm,
pre/.style={>=stealth,ultra thick,line width = 1.4pt}]
 \begin{scope}
    \tikzstyle{format} = [circle, inner sep=2pt,draw, thick, circle, line width=1.4pt, minimum size=3mm]

 \node[format] (A) {$A$};
        \node[format, left of=A] (B) {$B \atop 1$};
\node[format, right of=A] (Q) {$Q \atop 1$};
\node[format, below of=A] (Y) {$Y$};
\node[format, below of=B] (T) {$T \atop 2$};
\node[format, below of=Q] (R) {$R \atop 2$};                 

                 \draw (B) edge[pre, black] (A);
\draw (Q) edge[pre, black] (A);
 \draw (B) edge[pre, black, out=30, in=150] (Q);                                  

                 \draw (B) edge[pre, black] (T);
                 \draw (Q) edge[pre, black] (R);                 
                 \draw (T) edge[pre, black] (Y);       
 \draw (T) edge[pre, black, out=30, in=150] (R);                                  
 \draw (R) edge[pre, black] (Y);

 \end{scope} 
 \end{tikzpicture}
 }
\caption{An example of a directed acyclic graph in which the optimal minimum cardinality $L-N$ adjustment set is more efficient than the optimal minimum cost $L-N$ adjustment set.
}
\label{fig:Omin_better}
\end{figure}

\FloatBarrier

\newpage

\section{Appendix}
This section contains the proofs of all the results in the main paper,  as well as preliminary technical lemmas.

We will need the following lemmas in the proofs of Propositions \ref{prop:equiv_cut_sep} and \ref{prop:subset_implies_order}.
\begin{lemma}\label{lemma:equal_crossings}
 Let $Z$ be a minimal separator in $\mathcal{H}^{1}$.  Then \begin{equation}
(\mathfrak{d}(Z),  \overline{\mathfrak{d}(Z)}) =\lbrace (W^{\prime},  W^{\prime\prime}):  W\in Z\rbrace 
\nonumber
\end{equation}
\end{lemma}
\begin{proof}

We first show that $(\mathfrak{d}(Z),  \overline{\mathfrak{d}(Z)}) \supset \lbrace (W^{\prime},  W^{\prime\prime}):  W\in Z\rbrace$.
Since $Z$ is a minimal separator,  for any vertex $W\in Z$ there is a path connecting $W$ and $Y$ in $\mathcal{H}^{1}$ that does not intersect other vertices in $Z$,  and such a path corresponds to a directed path in $\mathcal{D}$ from $Y^{\prime\prime}$ to $W^{\prime}$ that does not intersect any other vertices $U^{\prime}$ or $U^{\prime\prime}$ for $U\in Z$. 
Hence,  if $W\in Z$ then $W^{\prime}\in\mathfrak{d}(Z)$ and $W^{\prime\prime}\notin \mathfrak{d}(Z)$ .
Thus $(\mathfrak{d}(Z),  \overline{\mathfrak{d}(Z)}) \supset \lbrace (W^{\prime},  W^{\prime\prime}):  W\in Z\rbrace $.

Next we prove that $(\mathfrak{d}(Z),  \overline{\mathfrak{d}(Z)}) \subset \lbrace (W^{\prime},  W^{\prime\prime}):  W\in Z\rbrace$. Take $(B,Q) \in (\mathfrak{d}(Z),  \overline{\mathfrak{d}(Z)})$,  we will show that $(B,Q) \in  \lbrace (W^{\prime},  W^{\prime\prime}):  W\in Z\rbrace$.  We have the following four cases to analyze.
\begin{itemize}
\item  Assume that $B=W^{\prime}$ for some $W\in Z$.  Due to how $\mathcal{D}$ was constructed,  necessarily $Q=W^{\prime\prime}$.  Thus $(B,Q) \in  \lbrace (W^{\prime},  W^{\prime\prime}):  W\in Z\rbrace$.
\item  Assume that $B=W^{\prime}$ for some $W\notin Z$.  Due to how $\mathcal{D}$ was constructed,  necessarily $Q=W^{\prime\prime}$.  Now,  since $W^{\prime}\in \mathfrak{d}(Z)$,  there exists a directed path $\delta$ in $\mathcal{D}$ from $Y^{\prime\prime}$ to $U^{\prime}$ for some $U\in Z$,  such that $\delta$ does not intersect $X^{\prime}$ or $X^{\prime\prime}$ for $X\in Z\setminus \lbrace U \rbrace$ and such that $W^{\prime}$ lies on $\delta$.   But $\delta$ has to go through  $W^{\prime\prime}$ to reach $U^{\prime}$ and this implies that $W^{\prime\prime} \in \mathfrak{d}(Z)$,  which contradicts the assumption that $W^{\prime\prime}=Q \in \overline{\mathfrak{d}(Z)}$.
\item  Next consider the case that $B=W^{\prime\prime}$ for some $W\in Z$.   This cannot happen,  since as we argued before,  if $W\in Z$ then $W^{\prime}\in \mathfrak{d}(Z)$ and $W^{\prime\prime}\notin \mathfrak{d}(Z)$.
\item Finally,  consider the case $B=W^{\prime\prime}$ for some $W \notin Z$.   Then, due to how $\mathcal{D}$ was constructed,  $Q=U^{\prime}$ for some $U$.  Since $Q=U^{\prime} \notin \mathfrak{d}(Z)$ then $U \notin Z$.  Now,  since $W^{\prime\prime} \in \mathfrak{d}(Z)$, there exists a directed path $\delta$ in $\mathcal{D}$ from $Y^{\prime\prime}$ to $W^{\prime\prime}$,  and hence to $U^{\prime}$, that does not intersect $X^{\prime}$ or $X^{\prime\prime}$ for $X\in Z$. 
In particular,  this implies that there is a path $\eta$ connecting $U$ and $Y$ in  $\mathcal{H}^{1}$ that does not intersect any vertices in $Z$.   Let $\nu$ be a path connecting $U$ and $A$  in $\mathcal{H}^{1}$.  Since $Z$ is a separator in $\mathcal{H}^{1}$, $\nu$ has to intersect $Z$.  Let $R$ be the vertex in $Z$ that lies closest to $U$ in $\nu$.   The sub-path of $\nu$ that goes from $U$ to $R$ corresponds to a directed path $\kappa$ from $U^{\prime}$ to $R^{\prime}$ in $ \mathcal{D}$.  Joining $\delta$ and $\kappa$ we get a directed path from $Y^{\prime\prime}$ to $R^{\prime}$ that does not intersect $X^{\prime}$ or $X^{\prime\prime}$ for $X\in Z\setminus \lbrace R \rbrace$.  Since $U^{\prime}$ lies on that path,  we get that $Q=U^{\prime}\in \mathfrak{d}(Z)$,  which is a contradiction.
\end{itemize}
We have thus shown that $(B,Q) \in  \lbrace (W^{\prime},  W^{\prime\prime}):  W\in Z\rbrace$,  finishing the proof of the lemma.
\end{proof}

\begin{lemma}\label{lemma:finite_cap}
There exists a cut in $\mathcal{D}$ with finite capacity.
\end{lemma}
\begin{proof}
Let
$$
S = \lbrace Y^{\prime},  Y^{\prime\prime} \rbrace \cup \lbrace W^{\prime}: W\neq A \rbrace.
$$
This is a set of vertices of $\mathcal{D}$ that contains $Y^{\prime\prime}$ and does not contain $A^{\prime}$ and hence it is a cut.  Its capacity is given by the sum of the capacities of all internal edges,  except $Y^{\prime}\to Y^{\prime\prime}$ and $A^{\prime}\to A^{\prime\prime}$.  Since these edges all have finite capacity,  the capacity of $S$ is finite,  which is what we wanted to show.
\end{proof}

We are now ready to prove Proposition \ref{prop:equiv_cut_sep}.
\begin{proof}[Proof of Proposition \ref{prop:equiv_cut_sep}]
We begin with the proof of the first assertion.  $\mathfrak{d}(Z)$ contains $Y^{\prime\prime}$ by definition.  We will show it does not contain $A^{\prime}$,  which will prove that $\mathfrak{d}(Z)$ is a cut.  Suppose for the sake of contradiction that $A^{\prime} \in \mathfrak{d}(Z)$.  Then there exists a directed path in $\mathcal{D}$ from $Y^{\prime\prime}$ to $A^{\prime}$ that does not intersect any vertices $W^{\prime}$ or $W^{\prime\prime}$ for $W\in Z$.  This implies that there exists a path in $\mathcal{H}^{1}$ connecting $Y$ to $A$ that does not intersect $Z$,  which contradicts the assumption that $Z$ is a separator in $\mathcal{H}^{1}$.  Thus,  $\mathfrak{d}(Z)$ is a cut.  
The fact that $k(\mathfrak{d}(Z))=c(Z)$ follows from Lemma \ref{lemma:equal_crossings}.

Next,  we prove the second part of the proposition.   We will first prove that $\mathfrak{h}(S)$ is a separator.  If $Y$ and $A$ are not connected in $\mathcal{H}^{1}$ then $\mathfrak{h}(S)$ is trivially a separator. Suppose then that there exists a path $\delta$ that connects $Y$ and $A$ in $\mathcal{H}^{1}$.  Such a path corresponds to a directed path from $Y^{\prime\prime}$ to $A^{\prime}$ in $\mathcal{D}$.  Since $S$ has finite capacity,  any such path  must contain an edge $W^{\prime}\rightarrow W^{\prime\prime}$ for some $W^{\prime}\in S$.  This implies that $\delta$ intersects  $W\in \mathfrak{h}(S)$,  which is what we wanted to show.  The claim that $c(\mathfrak{h}(S))=k(S)$ follows immediately from the definition of $\mathfrak{h}(S)$. 


Next,  we prove part three of the proposition.  Let $Z$ be a minimum cost separator.  Then it is a minimal separator,  and thus by part one $\mathfrak{d}(Z)$ is a cut with $k\left\lbrace \mathfrak{d}(Z) \right\rbrace = c(Z)$.  Suppose,  for the sake of contradiction,  that $\mathfrak{d}(Z)$ is not a min-cut and hence that there exists a cut $S$  in $\mathcal{D}$ such that $k(S) < k\left\lbrace \mathfrak{d}(Z) \right\rbrace$.  Since $S$ has finite capacity,  part two of the proposition implies that $\mathfrak{h}(S)$ is a separator with $c\lbrace \mathfrak{h}(S)\rbrace =k(S)$.  But then $c\lbrace \mathfrak{h}(S)\rbrace=k(S)< k\left\lbrace \mathfrak{d}(Z) \right\rbrace = c(Z)$,  contradicting the assumption that $Z$ was a minimum cost separator.  Thus,  it must be that $ \mathfrak{d}(Z) $ is a min-cut.

Turn now to the proof of part four of the proposition.  Let $S$ be a min-cut.  By Lemma \ref{lemma:finite_cap},  $S$ has finite capacity.  Then part two of the proposition implies that $\mathfrak{h}(S)$ is a separator with $c\lbrace \mathfrak{h}(S)\rbrace = k(S) $. Suppose,  for the sake of contradiction,  that $\mathfrak{h}(S)$ is not a minimum cost separator,  and hence that there exists a minimum cost separator $Z$ in $\mathcal{H}^{1}$ that satisfies $c(Z) < c\lbrace \mathfrak{h}(S)\rbrace$.  Since $Z$ is a minimal separator,  part one of the proposition implies that $\mathfrak{d}(Z)$ is a cut with $k\left\lbrace \mathfrak{d}(Z) \right\rbrace = c(Z)$. But then $k\left\lbrace \mathfrak{d}(Z) \right\rbrace = c(Z) < c\lbrace \mathfrak{h}(S)\rbrace = k(S)$,  contradicting the assumption that $S$ was a min-cut.  Thus,  it must be that $\mathfrak{h}(S)$ is a minimum cost separator.

Finally,  we prove the fifth part of the proposition.  We begin by showing that $Z \subset \mathfrak{h} \left\lbrace \mathfrak{d} \left( Z\right)  \right\rbrace$.  Take $W\in Z$.  We showed in Lemma \ref{lemma:equal_crossings} that $W^{\prime}\in  \mathfrak{d} \left( Z\right)$ and $W^{\prime\prime}\notin  \mathfrak{d} \left( Z\right)$.  Thus $W \in \mathfrak{h} \left\lbrace \mathfrak{d} \left( Z\right) \right\rbrace$.  Now we will show that $Z \supset \mathfrak{h} \left\lbrace \mathfrak{d} \left( Z\right)  \right\rbrace$.  Take $W\in \mathfrak{h} \left\lbrace \mathfrak{d} \left( Z\right)  \right\rbrace$.  Then $W^{\prime}\in  \mathfrak{d} \left( Z\right)$ and $W^{\prime\prime}\notin  \mathfrak{d} \left( Z\right)$.   Assume,  for the sake of contradiction,  that $W\notin Z$.  Since $W^{\prime}\in  \mathfrak{d} \left( Z\right)$  there exists in $\mathcal{D}$ a directed path $\delta$ from $Y^{\prime\prime}$ to $U^{\prime}$ for some $U\in Z$,  such that $\delta$ does not intersect any vertices $X^{\prime}$ or $X^{\prime\prime}$ for $X\in Z \setminus \lbrace U \rbrace$ and such that $W^{\prime}$ lies on $\delta$.  But since $\delta$ reaches  $U^{\prime}$,  it has to go through $W^{\prime\prime}$ too,  implying that $W^{\prime\prime}\in  \mathfrak{d} \left( Z\right)$,  which is a contradiction.  Thus,  it must be that $W\in Z$.
\end{proof}

\begin{proof}[Proof of Proposition \ref{prop:subset_implies_order}]
Take $W \in \mathfrak{h}(S^{\prime}) \setminus \mathfrak{h}(S)$ and a path $\delta$ in $\mathcal{H}^{1}$ connecting $W$ to $Y$.  We need to show that $\delta$ intersects  $\mathfrak{h}(S)$.  Now,  in $\mathcal{D}$ there is a path corresponding to $\delta$ of the form
$$
Y^{\prime\prime} \to U_{1}^{\prime} \to U_{1}^{\prime\prime} \to \dots \to  U_{l}^{\prime} \to U_{l}^{\prime\prime} \to W^{\prime} \to W^{\prime\prime}.
$$
Since $S$ is a cut,  $Y^{\prime\prime} \in S$.  Since $S$ is a min-cut,  by Lemma \ref{lemma:finite_cap} it has a finite capacity,  and thus it must be that  $U_{1}^{\prime}\in S$,  because the edge $Y^{\prime\prime} \to U_{1}^{\prime}$ has infinite capacity.  
 If $U_{1}^{\prime\prime}\notin  S$ then $U_{1}\in \mathfrak{h}(S)$ and we are done.  If $U_{1}^{\prime\prime}\in  S$,  since $S$ has finite capacity and the edge $U_{1}^{\prime\prime} \to U_{2}^{\prime} $ has infinite capacity it must be that $U_{2}^{\prime}\in S$.  We now repeat the same argument as before.  If at some point we find that 
 $U_{j}^{\prime}\in S$ and $U_{j}^{\prime\prime}\notin S$ then $U_{j}\in \mathfrak{h}(S)$ and we are done.  Otherwise all of $U_{1}^{\prime},U_{1}^{\prime\prime},\dots,  U_{l}^{\prime},U_{l}^{\prime\prime}$ are in $S$.  We will show that this cannot happen.  Assume it does.  Since the edge $U_{l}^{\prime\prime} \to W^{\prime}$ has infinite capacity,  it must be that $W^{\prime} \in S$.  If $W^{\prime\prime} \in S$,  since $S\subset S^{\prime}$ we conclude that $W^{\prime}$ and $W^{\prime\prime}$ are both in $S^{\prime}$,  which contradicts the assumption that $W\in \mathfrak{h}(S^{\prime})$.  Hence it must be that  $W^{\prime\prime} \notin S$,  but this implies that $W\in \mathfrak{h}(S)$,  which is a contradiction.

\end{proof}

The following lemma is a straightforward consequence of well known results in the theory of flow networks.  We include it here for completeness sake,  since we will need it in the proof of Proposition \ref{prop:Sc_smallest}.
\begin{lemma}\label{lemma:misc_flow_results}
Let $S$ be a cut.  Then $S$ is a min-cut if and only if it holds that for all $e \in (S,\overline{S})$,  $f^{\ast}(e)=k(e)$ and for all $e \in (\overline{S},S)$, $f^{\ast}(e)=0$.
\end{lemma}
\begin{proof}

We begin by noting the following.   By Lemma 5.1 of \cite{even},  the total flow of $f^{\ast}$ satisfies
\begin{equation}
F^{\ast} = 	\sum\limits_{e\in (S,\overline{S})} f^{\ast}(e) - 	\sum\limits_{e\in (\overline{S},  S)} f^{\ast}(e).
\label{eq:flow_decomp}
\end{equation}
and for all edges $e$ it holds that
\begin{equation}
0\leq f^{\ast}(e)\leq k(e).
\label{eq:valid_flow}
\end{equation}

Assume first that $S$ is a min-cut.
By the max-flow min-cut theorem (see Theorem 5.1 of \cite{even}),  $F^{\ast}$ satisfies
\begin{equation}
F^{\ast} = \sum\limits_{e\in (S,\overline{S})} k(e).
\label{eq:maxflow_mincut}
\end{equation}
It follows from \eqref{eq:flow_decomp}, \eqref{eq:valid_flow} and \eqref{eq:maxflow_mincut} that if $e \in (S,\overline{S})$,  $f^{\ast}(e)=k(e)$ and if $e \in (\overline{S},S)$, $f^{\ast}(e)=0$,  which is what we wanted to show.

Now assume that that  if $e \in (S,\overline{S})$,  $f^{\ast}(e)=k(e)$ and if $e \in (\overline{S},S)$, $f^{\ast}(e)=0$.  Then, by  \eqref{eq:flow_decomp}, the total flow of 
$f^{\ast}$ satisfies \eqref{eq:maxflow_mincut}.  
The max-flow min-cut theorem then implies that $S$ is a min-cut,  which is what we wanted to show.
\end{proof}

\begin{proof}[Proof of Proposition \ref{prop:Sc_smallest}]
We first show that $S_c$ is a cut.  We need to show that $Y^{\prime\prime}\in S_{c}$ and $A^{\prime}\notin S_{c}$.  That $Y^{\prime\prime}\in S_{c}$  follows from the definition of $S_{c}$.  On the other hand,  since $f^{\ast}$ is a max-flow there can be no paths from $Y^{\prime\prime}$ to $A^{\prime}$ that are augmenting for $f^{\ast}$,  since if there were,  the total flow of $f^{\ast}$ could be increased.  Thus,  $A^{\prime}\notin S_{c}$.  

Next,  we show that $S_c$ is a min-cut.  By Lemma \ref{lemma:characterize_mincot},  it suffices to show that if $e \in (S_{c},\overline{S}_{c})$ then $f^{\ast}(e)=k(e)$ and if $e \in (\overline{S}_{c},S_{c})$ then $f^{\ast}(e)=0$.  Take $W\in S_{c}$ and $U\in \overline{S}_{c}$.  Since $W\in S_{c}$,  there exists a path $\delta$ from $Y^{\prime\prime}$ to $W$ that is augmenting for $f^{\ast}$.  Suppose $e=(W,  U)$ is an edge in $\mathcal{D}$.  
  Then $f^{\ast}(e)=k(e)$,  because if $f^{\ast}(e)<k(e)$ the path obtained by joining $\delta$ and $e$ would be a path from $Y^{\prime\prime}$ to $U$ that is augmenting for $f^{\ast}$,  implying that $U\in S_{c}$,  which is a contradiction.  Suppose that $e=(U,  W)$ is an edge in $\mathcal{D}$. Then $f^{\ast}(e)=0$,  because if $f^{\ast}(e)>0$ the path obtained by joining $\delta$ and $e$ would be a path from $Y^{\prime\prime}$ to $U$ that is augmenting for $f^{\ast}$,  implying that $U\in S_{c}$,  which is a contradiction.  We have thus shown that 
$S_c$ is a min-cut.

Now take any other min-cut $S$.  We will show that $S_c \subset S$.
Take $U \in S_{c}$.  We need to show that $U\in S$.  Since $U\in S_{c}$, there exists a path $\delta$ from $Y^{\prime\prime}$ to $U$ in $\mathcal{D}$ that is augmenting for $f^{\ast}$.  Suppose  the vertices in $\delta$ are
$
Y^{\prime\prime},  W_{1},  W_{2} \dots,  W_{l},  W_{l+1}=U.
$
Since $S$ is a cut,  we have that $Y^{\prime\prime} \in S$.  
We will show that $W_{i} \in S$ for all $i=1, \dots, l+1$ by induction. 
Let $e_1$ be the edge joining $Y^{\prime\prime}$ and $W_{1}$ in $\delta$.  Since $\delta$ is augmenting for $f^{\ast}$,  we have that if $e_1=(Y^{\prime\prime}, W_1)$ then $f^{\ast}(e_1)<k(e_1)$ whereas if 
$e_1=(W_1,  Y^{\prime\prime})$ then $f^{\ast}(e_1)>0$.  Since $S$ is a min-cut,  Lemma \ref{lemma:misc_flow_results} implies that $W_{1}\in S$.  Now,  suppose that for some $1 \leq i < l+1$ it holds that $W_{i}\in S$.  Let $e_{i+1}$ be the edge joining $W_{i}$ and $W_{i+1}$ in $\delta$. 
 Since $\delta$ is augmenting for $f^{\ast}$,  we have that if $e_{i+1}=(W_{i}, W_{i+1})$ then $f^{\ast}(e_{i+1})<k(e_{i+1})$ whereas if 
$e_{i+1}=(W_{i+1}, W_{i})$ then $f^{\ast}(e_{i+1})>0$.   Since $S$ is a min-cut,  Lemma \ref{lemma:misc_flow_results} implies that $W_{i+1}\in S$.  This finishes the proof of the proposition.
\end{proof}

\begin{proof}[Proof of Theorem \ref{theo:main_theo}]
By Proposition \ref{prop:Sc_smallest},  $S_c$ is a min-cut.  Thus, by part four of Proposition \ref{prop:equiv_cut_sep},  $O_c=\mathfrak{h}(S_c)$ is a minimum cost separator in $\mathcal{H}^{1}$.  Lemma \ref{lemma:characterize_mincot} implies that $O_c$ is a minimum cost $L-N$ adjustment set.  

Now,  let $Z$ be any other minimum cost $L-N$ adjustment set.  We will show that  $O_c \preceq_{L} Z$.   Lemma \ref{lemma:characterize_mincot} implies that $Z$ is a minimum cost separator in $\mathcal{H}^{1}$. 
By part three of Proposition \ref{prop:equiv_cut_sep}, $\mathfrak{d}(Z)$ is a min-cut in $\mathcal{D}$.  Thus,  Proposition \ref{prop:Sc_smallest}   implies that $S_c \subset \mathfrak{d}(Z)$.  Proposition \ref{prop:subset_implies_order} implies that $\mathfrak{h}(S_c) \unlhd_{\mathcal{H}^{1}}\mathfrak{h} \left\lbrace  \mathfrak{d}(Z) \right\rbrace$.  But part five of Proposition \ref{prop:equiv_cut_sep} establishes that $ \mathfrak{h} \left\lbrace  \mathfrak{d}(Z) \right\rbrace= Z$.  We have shown that
$O_c \unlhd_{\mathcal{H}^{1}} Z$,  which by Lemma \ref{lemma_order_implies_var} implies that $O_c \preceq_{L} Z$. This finishes the proof of the theorem.
\end{proof}

\bibliographystyle{apalike}
\bibliography{weighted}

\end{document}